\documentclass[11pt,letterpaper,oneside,english]{article}
\usepackage{fullpage} % smaller margins

\usepackage{mdwlist,enumerate}
\usepackage{datetime}
\makeatletter
\let\@font@warningori\@font@warning
\newcommand\shutup{\def\@font@warning##1{}}
\newcommand\youcanspeak{\let\@font@warning\@font@warningori}
\makeatother 
\usepackage{amsmath,amsbsy,bm,latexsym,amssymb}
\shutup
\usepackage{fourier}
\youcanspeak
\usepackage[scaled=0.875]{helvet}

\usepackage{graphics, subfigure, float}
\usepackage{fp, calc}
\usepackage[latin1]{inputenc}
\usepackage[american]{babel}
\usepackage[T1]{fontenc} % Important: Use before loading the fonts
\usepackage{amscd,amsthm}
\usepackage{algorithm}
\usepackage{verbatim, comment}
\usepackage{enumitem}
\usepackage[dvips]{graphicx,epsfig,color}
\usepackage{pst-all}
\usepackage{pstricks-add}
\theoremstyle{theorem}
\newtheorem{theorem}{Theorem}
\newtheorem*{theorem*}{Theorem}

\newtheorem{lemma}[theorem]{Lemma}
\newtheorem*{lemma*}{Lemma}

\newtheorem*{claim*}{Claim}

\newtheorem*{conjecture*}{Conjecture}

\newtheorem*{problem*}{Problem}
\newtheorem{definition}{Definition}
\newtheorem*{definition*}{Definition}

\newtheorem*{fact*}{Fact}

\theoremstyle{remark}

\newtheorem*{remark*}{Remark}

\providecommand{\setR}{\mathbb{R}}

\newcommand{\polylog}{\ensuremath{\textrm{polylog}}}
\newcommand{\poly}{\ensuremath{\textrm{poly}}}

\newcommand{\R}{{\mathbb{R}}}

\newcommand{\vol}{\operatorname{vol}}

\newcommand{\width}{\textsc{width}}%{\operatorname{width}}

\usepackage{todonotes}
\newcommand{\E}{\mathop{\mathbb{E}}}

\psset{arrowsize=6pt, labelsep=2pt, linewidth=1.5pt}
\makeatother
\usepackage{hyperref}

\usepackage{calrsfs} % now e.g. \pazocal{I} gives a nice I
\DeclareMathAlphabet{\pazocal}{OMS}{zplm}{m}{n}

\title{An Improved Deterministic Rescaling for Linear Programming Algorithms} %Make a better name...
\author{Rebecca Hoberg\thanks{Email: {\tt rahoberg@uw.edu}} \quad and \quad Thomas Rothvo{ss}\thanks{Email: {\tt rothvoss@uw.edu}. Supported by an Alfred P. Sloan Research Fellowship. Both authors supported by NSF grant 1420180 with title ``\emph{Limitations of convex relaxations in combinatorial optimization}''. File compiled on {\today, \currenttime}.} 
\vspace{2mm} \\ University of Washington, Seattle} 

\begin{document}

\maketitle
\thispagestyle{empty}

\begin{abstract}
The \emph{perceptron algorithm} for linear programming, arising from machine learning, has been around since the 1950s. While not a polynomial-time algorithm, it is useful in practice due to its simplicity and robustness. 
In 2004, Dunagan and Vempala showed that a \emph{randomized rescaling} turns the perceptron method into a polynomial time algorithm, and later Pe\~{n}a and Soheili gave a \emph{deterministic rescaling}. 
In this paper, we give a deterministic rescaling for the perceptron algorithm that improves upon the previous rescaling methods by making it possible to rescale much earlier. 
This results in a faster running time for the rescaled perceptron algorithm. 
We will also 
demonstrate that the same rescaling methods yield a polynomial time algorithm based on the \emph{multiplicative weights update} method.
%our algorithm corresponds to a
%multiplicative weights update method that has an integrated rescaling procedure whenever 
%the progress drops below a certain threshold. 
This draws a connection to an area that has received a lot of recent
attention in theoretical computer science. %\rem{T: maybe phrase it like this...} 
%In addition, we show that the same rescaling methods can be used to make the \emph{multiplicative weights update} method into a polynomial time algorithm for linear programming, with similar running time guarantees. \rem{B: Could say something like ``We show nice connections between our rescaling method and the potential function used in MWU"}
\end{abstract}

\setcounter{page}{1}
\section{Introduction} % TEST TEST TEST

One of the central algorithmic problems in theoretical computer science as well as 
in more practical areas like operations research is finding the solution to a 
\emph{linear program}
\begin{equation} \label{eq:StandardFormLP}
  \max\{ c^Tx \mid Ax \geq b \}
\end{equation}
where $A\in \R^{m\times n}$, $c\in \R^n$ and $b\in \R^m$.
On the theoretical side, linear programming relaxations are the backbone for many approximation algorithms~\cite{DesignOfApproximationAlgorithms-ShmoysWilliamson,ApproximationAlgorithmsBook-Vazirani2001}. On the practical side, many real-world problems  
can either be modeled as linear programs or they can be modeled
at least as integer linear programs; the latter ones are then solved using 
Branch \& Bound or Branch \& Cut methods. Both of these methods rely on 
repeatedly computing solutions to linear programs~\cite{IntegerProgramming-ConfortiCornuejolsZambelliBook2014}. 

%A linear program is an optimization problem which maximizes a linear objective function subject to linear inequality constraints. In standard form, it takes the form $$\max_{\substack{Ax\le b\\x\ge 0}}c^Tx$$ for some matrix $A\in \R^{m\times n}$, $c\in \R^n$ and $b\in \R^m$.

%Linear programs have a wide array of applications, and so finding efficient algorithms is of great practical importance. 
The first algorithm for solving linear programs was the \emph{simplex method} due to Dantzig~\cite{dantzigsimplex}. While the method performs well in practice --- and is still 
the method of choice today --- for almost any popular pivoting rule one can
construct instances where the algorithm takes exponential time~\cite{kleeminty}.
%For a while it was an open question whether a polynomial time algorithm even existed. 
In 1979, Khachiyan~\cite{khachiyanellipsoid,schrijverlp} developed the first polynomial-time algorithm. However, despite the desirable theoretical properties, Khachiyan's \emph{ellipsoid method} turned out to be too slow for practical applications.

In the 1980s, \emph{interior point methods} were developed which were efficient in theory and in practice. Karmarkar's algorithm has a running time of $O(n^{3.5}L)$, where $L$ is the number of bits in the input~\cite{karmarkaripm}. Since then, there have been many further improvements in interior point methods. As recently as 2015, it was shown that there is an interior-point method using only $\tilde{O}(\sqrt{\operatorname{rank}(A)} \cdot L)$\footnote{The $\tilde{O}$-notation suppresses any $\polylog(m,n)$ terms.} many iterations; this upper bound essentially matches known lower bound barriers~\cite{leesinford}. 
%\rem{B: I guess this isn't quite the right definition of $\tilde{O}$, but this is what I had thought up until now. I guess it works, but is nonstandard.}

A common way to find a polynomial-time linear programming algorithm is with a greedy type procedure along with periodic rescaling \cite{Dadush:2016:RCD:2964664.2964667}.
One famous example of this is the \emph{perceptron algorithm}~\cite{RelaxationMethod-Agmon1954}, which we will focus on in this paper.
Instead of solving \eqref{eq:StandardFormLP} directly, this method finds a 
feasible point in the \emph{open polyhedral cone}
\begin{equation} \label{eq:ConicLP}
  P = \{ x \in \setR^n \mid Ax > \bm{0}\}
\end{equation}
where $A \in \setR^{m \times n}$ -- using standard reductions one can interchange 
the representations \eqref{eq:StandardFormLP} and \eqref{eq:ConicLP} with at most a linear overhead. 
The classical perceptron algorithm starts at the origin
and iteratively walks in the direction of any violated constraint. 
In the worst case this method is not polynomial time, but it is still useful due to its simplicity and robustness \cite{RelaxationMethod-Agmon1954}.
In 2004, Dunagan and Vempala~\cite{rescalingDV} showed that using a randomized rescaling 
procedure, the algorithm can be modified to find a point in \eqref{eq:ConicLP} in polynomial time. 
Explicitly, their algorithm runs in time $\tilde{O}(mn^4\log{\frac{1}{\rho}})$, where $\rho>0$ is the radius of the largest ball in the intersection of $P$ with the unit ball $B:=B(0,1)$.
 %(See Figure \ref{fig:reduction}.)
A deterministic rescaling procedure was provided by Pe{\~n}a and Soheili in \cite{SmoothPerceptron-PenaSoheili-MathProg2016}. Their algorithm
%was improved by Pe{\~n}a and Soheili~\cite{SmoothPerceptron-PenaSoheili-MathProg2016} using 
uses an improved convergence of the perceptron algorithm based on \emph{Nesterov's smoothing technique}~\cite{ExcessiveGapTechnique-Nesterov-SIAMJO2005, smoothperceptronSP}. Overall, their algorithm takes time $\tilde{O}(m^2n^{2.5}\log\frac{1}{\rho})$.

Another classical LP algorithm that we will discuss in this paper is based on a very general algorithmic framework called the \emph{multiplicative weights update} (MWU) method. 
In its general form one imagines having $m$ \emph{experts} who each incur some \emph{cost} in a sequence of iterations. % each expert incurs a \emph{cost} (or \emph{payoff} depending on the viewpoint). 
In each iteration we have to select a convex combination of experts so that
the expected cost is minimized, where we only have information on the past costs. 
The MWU method initially gives all experts the same weight and in each iteration the weight of 
expert $i$ is multiplied by 
$\exp(-\varepsilon \cdot \textrm{cost incurred by expert }i)$ where $\varepsilon$ is some
parameter. Then on average, the convex combination given by the weights will be nearly as good
 as the cost incurred by the best expert. %If the algorithm runs for
%$T = O(\frac{\ln(m)}{\varepsilon^2})$ iterations, the total cost iteration is
%at most $(1+\varepsilon) \cdot OPT + ?$, where $OPT$ is the cost of the best experts.
MWU is an \emph{online algorithm} that does not need to know
the costs in advance, and it has numerous applications in machine learning, economics and theoretical computer science.  
In fact, MWU has been reinvented many times under different names in the literature.
Recent applications in theoretical computer science include 
finding fast approximations to maximum flows~\cite{ApproximationOfMaxFlowInUndirectedGraphs-CKMST-STOC2011}, 
multicommodity flows~\cite{MulticommodityFlow-via-MWU-GargKoenemannSICOMP2007,ApproxMultiCommodityFlows-Madry-STOC2010}, solving LPs~\cite{PackingCoveringLPs-PlotkinShmoysTardos1995}, 
and solving semidefinite programs~\cite{FastAlgorithmsSDPviaMWU-AroraHazanKale-FOCS2005}.
We refer to the survey of Arora, Hazan and Kale~\cite{MWU-Survey-Arora-HazanKale2012} for a detailed overview. 

When we apply the MWU framework to linear programming, the experts correspond to the linear constraints. 
Suppose we use this method to find a valid point in $P=\{x:Ax>\bm{0}\}$ where $\|A_i\|_2 = 1$ for every row $A_i$\footnote{Notice that normalizing the rows does not affect the feasible region.}. % and $P\cap B$ contains a ball of radius $\rho$. 
At iteration $t$, the cost associated with expert $i$ will be $\langle A_i,p^{(t)}\rangle$ for some vector $p^{(t)}$.
Therefore the weight of expert $i$ at time $T$ will be $e^{-\langle A_i,x\rangle}$ where $x=\sum_{t=1}^T\varepsilon^{(t)} p^{(t)}$.
The analysis of MWU consists of bounding the sum of the weights, which in this case is given by the \emph{potential function} $\Phi(x)=\sum_{i=1}^me^{-\langle A_i,x\rangle}$.
If we choose the update vector $p^{(t)}$ to be a weighted sum of constraints at every iteration, notice that the resulting walk in $\setR^n$ corresponds to gradient descent on $\Phi$ -- in this case MWU terminates in $\tilde{O}(\frac{1}{\rho^2})$ iterations.
However, $\rho$ need not be polynomial in the input size, and in fact this method is not polynomial time in the worst case. 
%
%MWU is a general framework that has found repeated use in theoretical computer science. Recent applications include 
%finding fast approximations to maximum flows~\cite{ApproximationOfMaxFlowInUndirectedGraphs-CKMST-STOC2011}, 
%multicommodity flows~\cite{MulticommodityFlow-via-MWU-GargKoenemannSICOMP2007,ApproxMultiCommodityFlows-Madry-STOC2010}, solving LPs~\cite{PackingCoveringLPs-PlotkinShmoysTardos1995}, 
%and solving semidefinite programs~\cite{FastAlgorithmsSDPviaMWU-AroraHazanKale-FOCS2005}.
%We refer to the survey of Arora, Hazan and Kale~\cite{MWU-Survey-Arora-HazanKale2012} for a detailed overview. 
 %to design faster algorithms for central problems like max flow, linear programming and even semidefinite programming. 
%The outline for our algorithm is similar, but while Pe\~{n}a and Soheili require $\Delta=O(\frac{1}{m\sqrt{n}})$, we are able to rescale with $\Delta=O(\frac{1}{n})$.
%
%(Say something about MWU algorithm and what we do.)

\subsection{Our contribution}

%As in \cite{SmoothPerceptron-PenaSoheili-MathProg2016}, our rescaled perceptron algorithm will take the following form.

For reference, the general form for the rescaled LP algorithms we will present in this paper is given in Algorithm \ref{alg:fullalg}.

\begin{algorithm}
\caption{\label{alg:fullalg}}%Our Algorithm}
%\noindent Let $H=I$.\\
\noindent FOR $\tilde{O}(n\log\frac{1}{\rho})$ phases DO:
\begin{enumerate} %\vspace{-2mm}
\item[(1)] {\bf Initial phase:} 
%  \begin{itemize}[nolistsep]
Either find $x\in P$ or provide a $\lambda\in \setR^m_{\ge 0}$, $\|\lambda\|_1=1$ with $\|\lambda A\|_2\le \Delta$. 
%$\|\sum_{i=1}^m\lambda_iA_i\|_2\le \Delta$. %\\Details in Appendix B.
%\end{itemize}
\item[(2)] {\bf Rescaling phase:} %\vspace{-2mm}
%  \begin{itemize}[nolistsep]
Find an invertible linear transformation $F$ so that $\vol(F(P)\cap B)$ is a constant fraction larger than $\vol(P\cap B)$. Replace $P$ by $F(P)$. % $\|\lambda$\|\sum_{i=1}^m\lambda_iA_i\|_2\le \frac{\beta}{n}$, 
   %\vspace{-2mm} %Notice that we can perform at most $O(n\log{\frac{1}{\varepsilon}})$ such rescalings on $P$.
%  \end{itemize}
\end{enumerate}
\end{algorithm}

%Several algorithms already exist for the initial phase of the algorithm. 
%For instance, \cite{SmoothPerceptron-PenaSoheili-MathProg2016} shows that a smooth variant of the perceptron algorithm will run in time $O(\frac{mn}{\Delta})$. 

Our technical and conceptual contributions are as follows:

\begin{itemize}

\item[(1)] \emph{Improved rescaling:} We design a rescaling method that applies for a parameter of
$\Delta=\Theta(\frac{1}{n})$, which improves over the threshold $\Delta=\Theta(\frac{1}{m\sqrt{n}})$ required by \cite{SmoothPerceptron-PenaSoheili-MathProg2016}. 
This results in a smaller number of iterations that are needed 
per phase until one can rescale the system.

\item[(2)] \emph{Rescaling the MWU method:} 
We show that in $\tilde{O}(1/\Delta^2)$ iterations the MWU method can be made to implement the initial phase of Algorithm \ref{alg:fullalg}.
The idea is that if gradient descent is making insufficient progress then the gradient must have small norm, and from this we can extract an appropriate $\lambda$.
%We show that in each phase, either gradient descent %the MWU phase 
%will terminate within $O(\log{m}/\Delta^2)$ iterations %$\Phi$ will decrease by a good amount 
%or else we can find a convex combination of the constraints $A_1,...,A_m$ with Euclidean norm at most $\Delta$. In the context of the perceptron algorithm it was proved in \cite{SmoothPerceptron-PenaSoheili-MathProg2016} that 
%$\Delta=\Theta(\frac{1}{m\sqrt{n}})$ suffices to rescale and increase the volume of $P \cap B$
%by a constant factor. 
In particular, combining this with our rescaling method, we obtain a polynomial time LP algorithm based on MWU.

\item[(3)] \emph{Faster gradient descent:} The standard gradient descent approach terminates in at most $\tilde{O}(1/\Delta^2)$ iterations,
which matches the first approach in \cite{SmoothPerceptron-PenaSoheili-MathProg2016}. 
The more recent work of Pe{\~n}a and Soheili~\cite{smoothperceptronSP} uses \emph{Nesterov's smoothing technique}
to bring
the number of iterations down to a linear term of $\tilde{O}(1/\Delta)$. We prove that 
essentially the same speedup can be obtained without modifying the
objective function
by projecting the gradient on a significant eigenspace of the Hessian.

\item[(4)] \emph{Computing an approximate John ellipsoid:} For a general convex body $K$, computing a John ellipsoid is equivalent to finding a linear transformation so that $F(K)$ is well rounded.
For our unbounded region $P$, our improved rescaling algorithm gives a linear transformation $F$ so that $F(P)\cap B$ is well-rounded. %Due to space constraints, we postpone these results until Appendix~B. %
%See Appendix~B for details. %It is important to note here that the improved rescaling is necessary here.
\end{itemize}

\section{Rescaling of the Perceptron Algorithm}

In this section we fix an initial phase for Algorithm \ref{alg:fullalg} -- in particular, the paper of Pe\~{n}a and Soheili gives a smooth variant of the perceptron algorithm that achieves the following guarantee: % of the ``perceptron phase" in time $\tilde{O}(\frac{mn}{\Delta})$. 
\begin{lemma}[\cite{SmoothPerceptron-PenaSoheili-MathProg2016}] \label{lem:perceptronphase}
In time $\tilde{O}(\frac{mn}{\Delta})$, either the smooth perceptron phase outputs $x\in P$ or it gives $\lambda\in \setR^m_{\ge 0}$ with $\|\lambda\|_1=1$ and $\|\lambda A\|_2\le \Delta$.
\end{lemma}

We then focus on the rescaling phase of the algorithm. Our main result is that we are able to rescale with $\Delta=O(\frac{1}{n})$.

\begin{lemma}\label{lem:rescaling}
Suppose $\lambda\in \setR^m_{\ge 0}$ with $\|\lambda\|_1=1$ and $\|\lambda A\|_2\le O(\frac{1}{n})$.
Then in time $O(mn^2)$ we can rescale $P$ so that $\vol(P\cap B)$ increases by a constant factor.
\end{lemma}

We introduce two new rescaling methods that achieve the guarantee of Lemma \ref{lem:rescaling}. First we show that we can extract a thin direction by sampling rows of $A$ using a \emph{random hyperplane}. The linear transformation that scales $P$ in that direction, corresponding to a \emph{rank-1 update}, will increase $\vol(P\cap B)$ by a constant factor.

Next we give an alternate rescaling which is no longer a rank-1 update but which has the potential to increase $\vol(P\cap B)$ by up to an exponential factor under certain conditions.
In addition, if we take an alternate view where the cone $P$ is left invariant and instead update the underlying \emph{norm}, we see that this rescaling consists of adding a scalar multiple of a particular Hessian matrix to the matrix defining the norm.
We also believe that this view is the right one to make potential use of the \emph{sparsity} of the underlying matrix $A$, which would be a necessity for any practically relevant LP optimization method. 

Combining Lemmas \ref{lem:perceptronphase} and \ref{lem:rescaling} gives us the following theorem:
\begin{theorem}\label{thm:maintheoremperceptron} 
There is an algorithm based on the perceptron algorithm
that finds a point in $P$ in time $\tilde{O}(mn^3\log(\frac{1}{\rho}))$.
\end{theorem}

%In addition, if $\Delta$ is small enough, their algorithm finds a linear transformation that allows them to rescale.
%
%However, their rescaling requires $\Delta=O(\frac{1}{m\sqrt n})$. 
%In particular, this gives a way to run the first step of Algorithm \ref{alg:fullalg}. For the second step,
%\cite{SmoothPerceptron-PenaSoheili-MathProg2016} gives a rescaling, but it requires $\Delta=O(\frac{1}{m\sqrt{n}})$.
%In this section we present two such rescalings that will work for $\Delta=O(\frac{1}{n})$. This results in fewer iterations needed for the perceptron phase, and hence a faster running time for the algorithm.

%In this section we present two methods for finding such a rescaling.
%First, mirroring the arguments of \cite{}, we note that it is enough to find a direction in which the polytope is thin. 
%By sampling rows of $A$ using a \emph{random hyperplane}, we are able to find such a thin direction, and hence a rank-one linear transformation to increase $\vol(P\cap B)$.
%
%Second, we give a more general rescaling 
\subsection{Rescaling Using a Thin Direction}

In this section we will show how we can rescale by finding a direction in which the cone is \emph{thin}  -- see Figure~\ref{fig:widthpicture} for a visualization. First we give the formal definition of width.

\begin{definition} Define the width of the cone $P$ in the direction $c \in \setR^n \setminus \{ {\bm{0}} \}$ as
$$\width(P,c)=\frac{1}{\|c\|_2}\max_{x\in P\cap B} |\langle c,x\rangle |.$$
\end{definition}
\begin{figure}
\begin{center}
\psset{unit=2.5cm}
\begin{pspicture}(-1,-1)(1,1)
\pscircle[fillstyle=solid,fillcolor=black!10!white](0,0){1}
%\pswedge(-0.97,0.05){2}{-10}{20}
\pnode(0.93,0.66){A2}
\pnode(0.95,-0.36){B2}
\pscustom[fillstyle=vlines,linecolor=darkgray,hatchcolor=gray]{
\psline(0,0)(A2)
\nccurve[angleA=-75,angleB=65]{A2}{B2}
\psline(B2)(0,0)
}
%\psplot[algebraic=true]{0.943}{1}{2*sqrt(1-x^2)}
\pswedge[fillstyle=solid,fillcolor=black!50!white,opacity=0.5](0,0){1}{-10}{20}
\cnode*(0,0){2.5pt}{origin}
\psline[linestyle=dashed](-1.2,0.34)(1.2,0.34)
\psline[linestyle=dashed](-1.2,-0.34)(1.2,-0.34)
\pnode(-0.2,0.34){c1} \pnode(-0.2,0.7){c2} \ncline{->}{c1}{c2} \nbput[labelsep=2pt]{$c$}
\pnode(0,0.34){c3} \ncline{<->}{origin}{c3} \naput[labelsep=2pt]{$\frac{1}{3\sqrt{n}} \geq \textsc{Width}(P,c)$}
\rput[c](0.7,0.05){$P \cap B$}
\cnode*(0.93,0.33){2.5pt}{A}
\cnode*(0.95,-0.18){2.5pt}{B}
\rput[l](1.0,0.7){$ F(P \cap B)$}
\nput[labelsep=2pt]{-90}{origin}{$0$}
\end{pspicture}
\caption{Visualization of width and the rescaling operation \label{fig:widthpicture}
}
\end{center}
\end{figure}
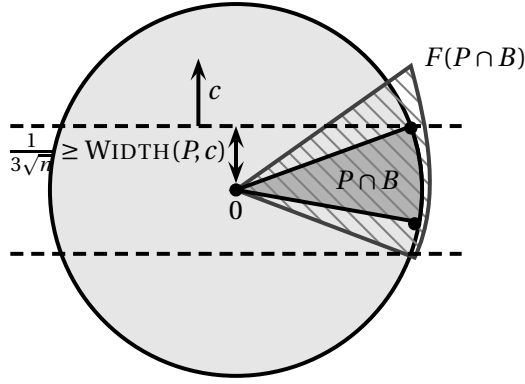

As described in \cite{SmoothPerceptron-PenaSoheili-MathProg2016}, we will now show that stretching $P$ in a thin enough direction increases the volume of $P\cap B$ by a constant factor.
%Now let's see how the volume of $P\cap B$ changes if we stretch $P$ in a skinny enough direction.
We reproduce the argument of~\cite{SmoothPerceptron-PenaSoheili-MathProg2016} here for the sake of
completeness: 
\begin{lemma}[\cite{SmoothPerceptron-PenaSoheili-MathProg2016}]  \label{lem:volincrease}
%Let $P \subseteq \{ x \in \setR^n \mid Ax \geq 0 \}$ be a cone and let $B := \{ x \in \setR^n \mid \|x\|_2 \leq 1\}$ be the unit ball. 
Suppose that there is a 
direction $c \in \setR^n \setminus \{ \bm{0} \}$ with $\width(P,c) \leq \frac{1}{3\sqrt{n}}$.
Define $F : \setR^n \to \setR^n$ as the linear
map with $F(c) = 2c$ and $F(x) = x$ for all $x \perp c$. Then
\[
 \textrm{vol}(F(P) \cap B) \geq \frac{3}{2} \cdot \textrm{vol}(P \cap B).
\]
\end{lemma}

\begin{proof}
We may assume that $\|c\|_2 = 1$.
Since $\det(F)=2$, we know that $\vol(F(P\cap B))=2\vol(P\cap B)$.
%But a priori intersecting this with the unit ball could decrease the volume significantly. 
%Since $c$ was a thin direction, no vector of $F(P\cap B)$ is longer than $1+\frac{1}{6n}\le e^{1/6n}$.
Now suppose that $x \in P \cap B$ and write it as $x = x' + \left<c,x\right> \cdot c$ where $x' \perp c$. 
Then $\|F(x)\|_2^2 = \|x' + 2\left<c,x\right> \cdot c\|_2^2 = \|x'\|_2^2 + 4\left<c,x\right>^2 \le \|x\|_2^2 + 3 \cdot {\textsc{width}}(P,c)^2 \leq 1+\frac{1}{3n}$ and taking square roots gives $\|F(x)\|_2 \leq 1+\frac{1}{6n}\leq e^{1/6n}$.
In particular, we know that $F(P\cap B)\subseteq e^{1/6n} \cdot F(P) \cap B$, and so
 we have \begin{equation*}\vol(F(P)\cap B)\ge (e^{1/6n})^{-n}\cdot \vol(F(P\cap B))\ge \frac{3}{4}\vol(F(P\cap B) =\frac{3}{2}\vol(P\cap B).\end{equation*}
\end{proof}

Explicitly, assuming $\|c\|_2=1$, Lemma \ref{lem:volincrease} updates our constraint matrix to $A(I-\frac{1}{2}cc^T)$. In particular, we apply a \emph{rank-1 update} to the constraint matrix.
Given a solution $x$ to these new constraints, a solution to the original problem can be easily recovered as $(I-\frac{1}{2}cc^T)x$.

It remains to argue how one can extract a thin direction for $P$, given a convex combination 
$\lambda$ so that $\|\lambda A\|_2$ is small. Here we will significantly improve over
the bounds of~\cite{SmoothPerceptron-PenaSoheili-MathProg2016} which require $\|\lambda A\|_2 \leq O(\frac{1}{m\sqrt{n}})$.
%So in order to rescale we need a direction of small width. One useful bound on the width is in terms of $\|\sum_{i\in J}\lambda_iA_i\|$ and $\|\sum_{i=1}^m\lambda_iA_i\|$ where $J\subset [m]$ is some subset of constraints.
%Notice that we have exactly said that we want an element of small enough norm in $\conv(A_i)$ %with $\|y\|_2$ small.
%One way to get this is the following.
We begin by a new generic argument to obtain a thin direction:
\begin{lemma}\label{lem:widthsubset} For any non-empty subset $J\subseteq [m]$ 
of constraints one has %has is a subset of constraints so that $\|\sum_{i\in J}\lambda_iA_i\|>0$. Then
%$$\width(P, \sum_{i\in I}\lambda_iA_i)\le \frac{\|A^T\lambda\|}{\rho}.$$ don't know which is better
$$\width\Big(P, \sum_{i\in J}\lambda_iA_i\Big)\le \frac{\|\sum_{i=1}^m\lambda_iA_i\|_2}{\|\sum_{i\in J}\lambda_iA_i\|_2}.$$
\end{lemma}
\begin{proof}
First, note that by the full-dimensionality of $P$, we always have 
$\|\sum_{i \in J} \lambda_i A_i\|_2 > 0$.
By definition of width, we can write 
$$\width\Big(P,\sum_{i\in J}\lambda_iA_i\Big)=\frac{1}{\|\sum_{i\in J}\lambda_iA_i\|_2}\max_{x\in P\cap B}\langle \sum_{i\in J}\lambda_iA_i,x\rangle.$$
Now, we know that $\langle A_i,x\rangle \geq 0$ for all $x\in P$ and so
\[
\max_{x\in P\cap B}\langle \sum_{i\in J}\lambda_iA_i,x\rangle\le  \max_{x \in B} \langle\sum_{i=1}^m \lambda_iA_i,x\rangle = \|\lambda A\|_2
\]
%\max_{y\in B}\sum_{i=1}^m\lambda_i\langle A_i,x\rangle =
and the claim is proven. 
%Finally, we rewrite this as $\max_{y\in B}\langle\sum_{i=1}^m\lambda_i A_i,x\rangle$
%and so Cauchy-Schwarz gives the claim.
\end{proof}

So in order to find a direction of small width, it suffices to find a subset $J\subset [m]$ with $\|\sum_{i\in J}\lambda_iA_i\|_2$ large.
Implicitly, the choice that Pe{\~n}a and Soheili~\cite{SmoothPerceptron-PenaSoheili-MathProg2016}
make is to select $J=\{i_0\}$ for $i_0 \in [m]$ maximizing $\lambda_{i_0}$. This
approach gives a bound of $\|\sum_{i\in J}\lambda_iA_i\|_2\ge \frac{1}{m}$.
We will now prove the asymptotically optimal bound\footnote{It suffices here to consider the trivial example with $\lambda_1=\ldots = \lambda_n = \frac{1}{n}$ and $A_i = e_i$ being the standard basis. Then $\|\sum_{i \in J} \lambda_i A_i\|_2 \leq \frac{1}{\sqrt{n}}$ for any subset $J$. The optimality of our rescaling can also be seen since the cone in the last iteration is $\tilde{O}(n)$-well rounded, which is optimal up to $\tilde{O}$-terms.} using a random hyperplane:
\begin{lemma}\label{lem:gaussian}
Let $\lambda \in \setR_{\geq 0}^m$ be any convex combination and $A \in \setR^{m \times n}$
with $\|A_i\|_2 = 1$ for all $i$.
Take a random Gaussian $g$ and set $J := \{ i \in [m] \mid \left<A_i,g\right> \geq 0\}$.
Then with constant probability $\|\sum_{i \in J} \lambda_i A_i\|_2 \geq \frac{1}{4\sqrt{\pi n}}$.
\end{lemma}
\begin{proof} 
%Suppose $g$ is chosen from an $n$-dimensional gaussian distribution, and let $v=\frac{g}{\|g\|_2}$. Then 
We set $v := \frac{g}{\|g\|_2}$. Since $v$ is unit vector we can lower bound the
length of $\|\sum_{i \in J} \lambda_i A_i\|_2$ by measuring the projection on $v$ and obtain
$
 \|\sum_{i \in J} \lambda_iA_i\|_2 \geq \sum_{j \in J} \lambda_i\left<A_i,v\right>.
$
By symmetry of the Gaussian it then suffices to argue that $\sum_{i=1}^m \lambda_i |\left<A_i,v\right>| \geq \frac{1}{2\sqrt{\pi n}}$.
%The idea is that $v$ with $\sum_{i \in J} \lambda_i| <A_i,v>| \geq$
First we will show that for an appropriate constant $\alpha\in(0,1)$,
\begin{enumerate}
\item[(1)] $\Pr(\|g\|_2\ge \sqrt{2n})\le \frac{2}{n}$
%\item [(2)]$\Pr(\sum_i\lambda_i|\langle A_i,g\rangle |<\alpha)\le \frac{1}{2}$
\item [(2)]$\Pr(\sum_{i=1}^n\lambda_i|\langle A_i,g\rangle |<\sqrt\frac{1}{2\pi})\le \alpha$.
\end{enumerate}
Then, with probability at least $\gamma=\frac{1-\alpha}{2}$, we have 
$\sum_{i=1}^n\lambda_i|\langle A_i,v\rangle |\ge \frac{1}{2\sqrt{\pi n}}.$

For (1), notice that $\|g\|_2^2$ is just the chi-squared distribution
with $n$ degrees of freedom, and so it has variance $2n$ and mean $n$. 
Therefore Chebyshev's inequality tells us that 
 $\Pr\left[\|g\|_2^2 \ge 2n\right]\le \frac{2}{n}$.
Now, for all $i$, $\langle A_i,g\rangle$ is a normal random variable with mean $0$ and variance $1$, 
and so the expectation of its absolute value is $\sqrt\frac{2}{\pi}$.
Summing these up gives $\E\left[\sum_{i=1}^m\left|\langle \lambda_i A_i,g\rangle\right|\right]=\sqrt \frac{2}{\pi}.$
Moreover, $\sum_{i=1}^m\left|\langle \lambda_i A_i,g\rangle\right|$ is Lipschitz in $g$ with Lipschitz constant $1$, 
and so\footnote{Recall that a function $F : \setR^n \to \setR$ is \emph{Lipschitz} with Lipschitz constant 1 if 
$|F(x)-F(y)| \leq \|x-y\|_2$ for all $x,y \in \setR^n$. A famous concentration inequality by Sudakov, Tsirelson, Borell states that $\Pr[|F(g)-\mu| \geq t] \leq e^{-t^2/\pi^2}$, where $g$ is a random Gaussian and $\mu$ is the mean of $F$ under $g$.} 
$$\Pr\left(\sum_{i=1}^m\left|\langle \lambda_i A_i,g\rangle\right|<\sqrt{\frac{2}{\pi}} - t\right)\le e^{-t^2/\pi^2}.$$ 
Letting $t=\sqrt \frac{1}{2\pi}$ gives (2).
By a union bound, the probability either of these events happens is at most $\alpha+\frac{2}{n}$, and so with probability at least $\frac{1-\alpha}{2}$ neither occurs, which gives us the claim. 
\end{proof}
While the proof is probabilistic, one can use the method of \emph{conditional expectation}
to derandomize the sampling~\cite{opac-b1088796}. %\rem{T: One can cite the book of Alon and Spencer here.}
More concretely, consider the function $F(g) := \sum_{i=1}^m \lambda_i |\left<A_i,g\right>| - \frac{1}{10\sqrt{n}} \|g\|_2$. The proof of Lemma~\ref{lem:gaussian}
implies that the expectation of this function is at least $\Omega(1)$. Then we can find a desired vector $g=(g_1,\ldots, g_n)$ by choosing the coordinates one after the other
%the coordinates $g = (g_1,\ldots,g_n)$ one after the other 
so that the conditional expectation does not decrease.
%will give a vector as desired. %\rem{T: As we promised a deterministic rescaling, I thought I add some details here.}
%\begin{lemma}\label{lem:gaussian}
%Suppose $g$ is chosen from an $n$-dimensional gaussian distribution, and let $v=\frac{g}{\|g\|_2%}$. Then 
%%for some constant $\alpha>0$, we have that with probability at least $\frac{1}{3}$,
%%$$\sum_{i=1}^n\lambda_i|\langle A_i,v\rangle |\ge \frac{\alpha}{\sqrt n}.$$
%with probability at least $\gamma$, where $\gamma\in(0,1)$ an appropriate constant, we have 
%$$\sum_{i=1}^n\lambda_i|\langle A_i,v\rangle |\ge \frac{1}{2\sqrt{\pi n}}.$$
%\end{lemma} 
%For the proof, see Appendix D. 
%Now setting $\beta=\frac{1}{8\sqrt{\pi}}$, our rescaling phase will look as follows. %MAKE THIS BETTER!!
We are now ready to prove Lemma \ref{lem:rescaling}, which we restate here with explicit constants. %, with $\beta=\frac{1}{12\sqrt{\pi}}$.
%\begin{algorithm}[H]
%\caption{Rescaling Phase}
%Given a convex combination of constraints with $\|\sum_{i=1}^m\lambda_iA_i\|_2\le \frac{\beta}{n}$:
%%Given $\lambda$, $P$, 
%\begin{enumerate}
%\item Choose $v$ at random as in Lemma \ref{lem:gaussian} until we have $\sum_{i=1}^n\lambda_i|\langle A_i,v\rangle |\ge \frac{1}{2\sqrt{\pi n}}.$ 
%%\item Find unit vector $v$ with $$\sum_{i=1}^m\lambda_i|\langle A_i,v\rangle|\ge something$$
%\item Use $v$ to choose $J\subseteq [m]$ so that $\|\sum_{i\in J}\lambda_iA_i\|_2\ge \frac{1}{4\sqrt{\pi n}}$. %\Omega(\frac{1}{\sqrt{n}})$.
%\item Let $c=\frac{\sum_{i\in J}\lambda_iA_i}{\|\sum_{i\in J}\lambda_iA_i\|_2}$.
%%{\|\sum_{i\in J}\lambda_iA_i\|_2}$
%% \le O(\frac{\delta}{\sqrt{n}})$.
%and replace $A$ by $A(I-\frac{1}{2}cc^T)$. Normalize so $\|A_i\|_2=1$. 
%%$R$ by $(I-\frac{1}{2}cc^T)R$. 
%\end{enumerate}
%\end{algorithm}

%We can now prove Lemma \ref{lem:rescaling}, which we restate with the specific conditions on $\lambda$.
%\theoremstyle{theorem}
%\newtheorem*{lem:rescaling}{Lemma \ref{lem:rescaling}}
%\begin{lem:rescaling}
\begin{lemma}
%Let $\lambda\in \R_{\ge 0}^n$ with $\|\lambda\|_1=1$. 
%Suppose $\lambda\in \setR_{\ge 0}^m$ with $\|\lambda\|_1=1$, and the maximum eigenvalue of $M=\sum_{i=1}^m\lambda_iA_iA_i^T$ is $\alpha>0$.
%For a small enough constant $\beta$, i
Suppose $\lambda\in \setR_{\ge 0}^m$ with $\|\lambda\|_1=1$ and $\|\lambda A\|_2\le \frac{1}{12n\sqrt{\pi}}$.
Then in time $O(mn^2)$ we can rescale $P$ so that $\vol(P\cap B)$ increases by a constant factor.
%Given a convex combination $\lambda$ of constraint vectors with $\|\sum_{i=1}^n\lambda_iA_i\|_2\le \frac{1}{12n\sqrt{\pi}}$, %\frac{\beta}{n}$,
%%$\|y\|_2\le \lambda_{\textrm{max}}(M)\frac{c}{\sqrt n},$
%%$\|y\|_2\le \|M\|_2\frac{c}{\sqrt n},$
% we can rescale $P$ so that $\vol(P\cap B)$ increases by a constant factor in expected time $O(mn^{\omega-1})$.
%(Should maybe say something here about recoverability).
\end{lemma}
%\end{lem:rescaling}
\begin{proof}

Computing a random Gaussian and checking if it satisfies the conditions of Lemma \ref{lem:gaussian} takes time $O(mn)$. Since the conditions will be satisfied with constant probability, the expected number of times we must do this is constant.
Once the conditions are satisfied, finding 
a thin direction and rescaling can be done in time $O(n^3)$.
%a good subset $J$ takes time $O()$
%with $\|\sum_{i\in J}\lambda_iA_i\|_2\ge \frac{1}{4\sqrt{\pi n}}$ and rescale $A$ accordingly, all of which can be done in time $O(mn^{\omega-1})$.
%and rescaling takes time $O()$.
Lemmas \ref{lem:volincrease} and \ref{lem:widthsubset} guarantee we get a constant increase in the volume.
\end{proof}

%Combining Lemmas \ref{lem:perceptronphase} and \ref{lem:rescaling} completes the proof of Theorem \ref{thm:maintheoremperceptron}.

%\begin{theorem}
%Algorithm \ref{alg:fullalg} finds an element of $P$ in expected time $\tilde{O}(mn^3\log(\frac{1}{\rho}))$. 
%\end{theorem} %how to phrase "expected running time" bit? Because it's a lot stronger than that.

\subsection{Deterministic Multi-rank Rescaling} 
We now introduce an alternate linear transformation we can use to rescale. 
This is no longer a rank-1 update, but it is inherently deterministic along with other nice properties. 
For one thing, although we only guarantee constant improvement in the volume, under certain circumstances the rescaling can improve the volume by an exponential factor.
This transformation will also take a nice form when we change the view to consider rescaling the unit ball rather than the feasible region.
 
\begin{lemma}\label{lem:altrescaling} Suppose $\lambda\in \setR^m_{\ge 0}$, $\|\lambda\|_1=1$ and $\|\lambda A\|_2\le \frac{1}{10n}$. Let $M$ denote the matrix $\sum_{i=1}^m\lambda_iA_iA_i^T$ and suppose $0\le \alpha\le \frac{1}{\delta_{\textrm{max}}}$, where $\delta_{\max}=\|M\|_{\textrm{op}}$ denotes the maximal eigenvalue of $M$.
Define $F(x)=(I+\alpha M)^{1/2}x$.
Then $\vol(F(P)\cap B)\ge e^{\alpha/5}\vol(P\cap B)$.
\end{lemma}

\begin{proof} First notice that $M$ is symmetric positive semi-definite with trace $1$. Therefore the eigenvalues of $I+\alpha M$ take the form $1+\alpha \delta_i$ where $0\le \delta_i\le \alpha$ and $\sum_{i=1}^n\delta_i=1$. Note that since $\alpha\delta_i\le 1$, we can lower bound the eigenvalues by $1+\alpha \delta_i\ge e^{\alpha\delta_i/2}.$
Therefore $$\det(I+\alpha M)\ge \prod_{i=1}^ne^{\alpha\delta_i/2}=\exp\Big(\frac{\alpha}{2}\sum_{i=1}^n\delta_i\Big)=e^{\alpha/2}.$$
%=\prod_{i=1}^n(1+\alpha \delta_i)
In particular, $\det(F)\ge e^{\alpha/4}$.

So far we have shown that $\vol(F(P\cap B))$ is significantly larger than $\vol(P\cap B)$. However, the desired bound is on $\vol(F(P)\cap B)$, and so we need to ensure that we do not lose too much of the volume when we intersect with the unit ball. It turns out the bound on $\|\lambda A\|_2$ will allow us to do precisely this.

For any $x$, we get the bound
$$\|F(x)\|_2^2=x^T\Big(I+\alpha\sum_{i=1}^m \lambda_iA_iA_i^T\Big)x= \|x\|_2^2+\alpha \sum_{i=1}^m \lambda_i \langle A_i,x\rangle^2\le  \|x\|_2^2.$$
Now, if we assume that $x\in P\cap B$, this becomes
$$\|F(x)\|_2^2\le 1+\alpha \sum_{i=1}^m \lambda_i\langle A_i,x\rangle \le 1+\alpha\|\lambda A\|_2 \le 1+\frac{\alpha}{10n}.$$
%\le\|\sum\lambda_iA_i\|_2$. 
%In particular, this implies $\|F(x)\|_2^2
%=\|x\|_2^2+x^T\sum\lambda_iA_iA_i^Tx
%\le 1+\frac{1}{30n}$.
% \sum \lambda_i \langle A_i,x\rangle^2\le\|\sum\lambda_iA_i\|_2$.
The point is that every element of $F(P\cap B)$ has length at most $1+\frac{\alpha}{20n}$, and so intersecting with the unit ball will not lose more volume than shrinking by a factor of $1+\frac{\alpha}{20n}$. 
In particular, the volume decreases by at most $(1+\frac{\alpha}{20n})^{-n}\ge e^{-\alpha/20}$, and so we have %\rem{B: Can definitely shorten this a bit.}
%$$\vol(F(P)\cap B)=\vol(F(P\cap B)\cap B)\ge (1+\frac{\alpha}{20n})^{-n}\vol(F(P\cap B)).$$
%In particular, $(1+\frac{\alpha}{20n})^{-n}\ge e^{-\alpha/20}$, and so we have 
$$\vol(F(P)\cap B)\ge e^{-\alpha/20}\vol(F(P\cap B))\ge e^{-\alpha/20}\cdot e^{\alpha/4} \vol(P\cap B) \ge e^{\alpha/5}\cdot \vol(P\cap B).$$
\end{proof}

Note that one always has $\delta_{\max} \leq 1$ and hence in any case one can choose $\alpha \geq 1$.
Therefore if $\|\lambda A\|_2\le \frac{1}{10n}$, we get constant improvement in $\vol(P\cap B)$. 
In fact, if the eigenvalues of $M$ happen to be small, we could get up to exponential improvement.
This computation can be carried out in time $O(mn^2)$ and so Lemma \ref{lem:altrescaling} proves Lemma \ref{lem:rescaling} and hence Theorem \ref{thm:maintheoremperceptron}.

\subsection{An Alternate View of Rescaling}
Obviously instead of applying a linear transformation
to the cone $P$ itself, there is an equivalent view where instead one applies
a linear transformation to the unit ball. % that defines the gradient. 
We will now switch the view in the sense that we fix the
cone $P$, but we update the norm in each rescaling step so that the unit ball becomes more representative of $P$.

Recall that a symmetric positive definite matrix $H \in \setR^{n \times n}$ induces a \emph{norm}
$\|x\|_H := \sqrt{x^THx}$.
Note that also $H^{-1}$ is a symmetric positive definite matrix\footnote{An easy way to see this is to write 
$H = \sum_{j=1}^n \mu_j u_ju_j^T$ as the eigendecomposition of $H$. Then $H^{-1} = \sum_{j=1}^n \frac{1}{\mu_j} u_ju_j^T$ is the inverse; clearly all eigenvalues are positive and the inverse has the same spectrum as $H$.} 
and $\|\cdot\|_{H^{-1}}$ is the \emph{dual norm} of $\|\cdot \|_H$.
In this view we assume the rows $A_i$ of $A$ are normalized so that $\|A_i\|_{H^{-1}}=1$. 

%Let $\pazocal{B}_H$ denote the unit ball for the norm $\|\cdot \|_H$.
%We see that $\pazocal{B}_H = \{ x \in \setR^n \mid \|x\|_H \leq 1\} = \{ x \in \setR^n \mid x^THx \leq 1\}$
%from which we see that $\pazocal{B}_H$ is an \emph{ellipsoid}.
Let $B_H := \{ x \in \setR^n \mid \|x\|_H \leq 1\}$ be the unit ball for the norm $\|\cdot \|_H$. Note that $B_H$ is always an
\emph{ellipsoid}.
We will measure progress in terms of the fraction
of the ellipsoid $B_H$ that lies in the cone $P$, namely $\mu(H) := \frac{\textrm{vol}(B_H \cap P)}{\textrm{vol}(B_H)}$.
The goal of the rescaling step will then be to increase $\mu(H)$ by a constant factor.
Note that we initially have $\mu(H)=\mu(I)\ge \rho^n$, and at any time $0\le \mu(H)\le 1$, so we can rescale at most $O(n\log{\frac{1}{\rho}})$ times. 

In this view, Lemma \ref{lem:altrescaling} takes the following form:
%We now show that given a convex combination $\lambda$ with $\|\lambda A\|_{H^{-1}}\le \frac{1}{20n}$ we can update $H$ to improve $\mu(H)$ by a constant factor.
\begin{lemma}\label{lem:rescalingphaseH}
Let $H \in \setR^{n \times n}$ be symmetric with $H \succ 0$.
Suppose $\lambda \in \setR^{m}_{\geq 0}$ with $\|\lambda\|_1=1$ and $\|\lambda A\|_{H^{-1}} \leq \frac{1}{10n}$
and let $M := \sum_{i=1}^m \lambda_iA_iA_i^T$.
Let $0 \leq \alpha \leq \frac{1}{\delta_{\max}}$, where $\delta_{\max} := \|H^{-1}M\|_{\textrm{op}}$. Then for $\tilde{H} := H + \alpha M$ one has $\mu(\tilde{H}) \geq e^{\alpha / 5} \cdot \mu(H)$.
\end{lemma}

Algorithm \ref{alg:fullalgH} illustrates what the multi-rank rescaling looks like under the alternate view. Notice that the algorithm updates the norm matrix by adding a scalar multiple of the Hessian matrix of the MWU potential function discussed in Section 3.
Moreover, throughout the algorithm our matrix $H$ will have the form $I+\sum_{i=1}^mh_iA_iA_i^T$ for some $h_i\ge 0$. Note that this allows fairly compact representation as 
we only need $O(m)$ space to encode the coefficients $h_i$ that define the norm matrix.

\begin{algorithm}
\caption{\label{alg:fullalgH}}%Our Algorithm}
%\noindent Let $H=I$.\\
\noindent FOR $\tilde{O}(n\log\frac{1}{\rho})$ phases DO:
\begin{enumerate} %\vspace{-2mm}
\item[(1)] {\bf Initial phase:} 
%  \begin{itemize}[nolistsep]
Either find $x\in P$ or give $\lambda\ge 0$, $\|\lambda\|_1=1$ with $\|\lambda A\|_{H^{-1}}\le O(\frac{1}{n})$. 
%$\|\sum_{i=1}^m\lambda_iA_i\|_2\le \Delta$. %\\Details in Appendix B.
%\end{itemize}
\item[(2)] {\bf Rescaling phase:} %\vspace{-2mm}
%  \begin{itemize}[nolistsep]
Update $H:=H+\alpha M$, where $M=\sum_{i=1}^m\lambda_iA_iA_i^T$. 
%Rescale so that $\|A_i\|_{H^{-1}}=1$.
% $\|\lambda$\|\sum_{i=1}^m\lambda_iA_i\|_2\le \frac{\beta}{n}$, 
   %\vspace{-2mm} %Notice that we can perform at most $O(n\log{\frac{1}{\varepsilon}})$ such rescalings on $P$.
%  \end{itemize}
\end{enumerate}
\end{algorithm}

%update a PSD matrix $H$ in each phase so that the ellipsoid represented by $H$ becomes
%more and more representative of $P$. 

%The algorithm will then look as follows:
%\begin{algorithm*} 
%\caption{}
%For $\tilde{O}(n\log \frac{1}{\varepsilon})$ iterations:
%\begin{enumerate}
%\item Run a MWU phase to either return $x\in P$ or find a $\lambda$ that will allow us to rescale.
%\item %Given $\lambda$ satisfying certain conditions, r
%Rescale $H$ to increase $\mu(H)$ by a constant factor.
%\end{enumerate}
%\end{algorithm*}

%The MWU phase of Algorithm \ref{alg:fullalgH} will be similar to that of Algorithm \ref{alg:fullalg}, but with different norms.

%

%Note that one always has $\delta_{\max} \leq 1$ and hence in any case one can choose $\alpha \geq 1$.
%
%Therefore if $\|\lambda A\|_{H^{-1}}\le \frac{1}{10n}$, we get constant improvement in $\mu(H)$. 
%In fact, if the eigenvalues of $M$ happen to be small, we could get up to exponential improvement.
%This computation can be carried out in time $O(mn^{\omega-1})$ and so Lemmas \ref{lem:mwuphaseH} and \ref{lem:rescalingphaseH} give us the desired guarantee.
%

\section{Rescaling for the MWU algorithm} 

In this section we show that the same rescaling methods can be used to make the MWU method into a polynomial time algorithm for linear programming.

Recall that the MWU algorithm corresponds to gradient descent on a particular potential function.
First we show how we can apply rescaling to the standard gradient descent approach.
We then introduce a modified gradient descent, which speeds up the MWU phase.
Combining this with our rescaling step above gives us the following result:

\begin{theorem}\label{thm:maintheoremmwu}
There is an algorithm based on the MWU algorithm
that finds a point in $P$ in time \\$\tilde{O}(mn^{\omega+1}\log(\frac{1}{\rho}))$,
where $\omega\approx 2.373$ is the exponent of matrix multiplication.
\end{theorem}

%Finally, we note that our rescaling algorithm in fact gives an approximate John ellipsoid for the feasible region $P$.
%See Appendix~B for details.

\subsection{Standard Gradient Descent}
Consider the potential function $\Phi(x)=\sum_{i=1}^me^{-\langle A_i,x\rangle},$ where $\|A_i\|_2=1$ for all rows $A_i$.
Notice that $\Phi(0)=m$ and that if $\Phi(x)<1$ then $\langle A_i,x\rangle>0$ for all $i$, and hence $x\in P$. 
%The goal of the MWU method is therefore to decrease $\Phi(x)$, which it does by using gradient descent beginning at the origin. 
In this section we analyze standard gradient descent on $\Phi$, starting at the origin. 
Notice that the gradient takes the form $$\nabla \Phi(x)=-\sum_{i=1}^me^{-\langle A_i,x\rangle}A_i.$$
If we let $\lambda_i=\frac{1}{\Phi(x)}e^{-\langle A_i,x\rangle}$, we see that $\|\lambda\|_1=1$ and
$\lambda A=-\frac{\nabla \Phi(x)}{\Phi(x)}$.
In particular, if at any iteration this vector has small Euclidean norm, then we will be able to rescale.
It remains to show, therefore, that if this vector has large Euclidean norm, then we get sufficient decrease in the potential function.

\begin{lemma}\label{lem:phidecrease}
Suppose $x\in \setR^n$ and abbreviate $y=-\frac{\nabla\Phi(x)}{\Phi(x)}$.
Then %$$\Phi(x+\frac{1}{2}\lambda A)\le \Phi(x)\cdot (1-\frac{1}{4}\|\lambda A\|_2^2).$$
$$\Phi(x + \frac{1}{2}y) \leq \Phi(x) \cdot e^{-\|y\|_2^2/4}.$$
\end{lemma}

\begin{proof}
First note that since $\|\lambda\|_1=1$ and $\|A_i\|_2=1$, we know that $|\langle A_i,y\rangle|\le 1$ for all $i$. In our analysis we will also use the fact that for any $z\in \R$ with $|z|\le 1$ one has $e^z\le 1+z+z^2$. We obtain the following.

\begin{eqnarray*}
\Phi(x+\frac{1}{2}y) &=& \sum_{i=1}^m e^{-\langle A_i,x+\frac{1}{2} y\rangle}
= \sum_{i=1}^m e^{-\langle A_i,x\rangle }e^{-\frac{1}{2}\langle A_i,y\rangle}\\
&{\le} & \sum_{i=1}^m e^{-\langle A_i,x\rangle}(1-\frac{1}{2} \langle A_i,y\rangle +\frac{1}{4}\langle A_i,y\rangle ^2)\\
&=&\Phi(x)\cdot\sum_{i=1}^m\lambda_i(1-\frac{1}{2} \langle A_i,y\rangle +\frac{1}{4} y^TA_iA_i^Ty)\\
&\le & \Phi(x)\cdot (1-\frac{1}{4}\|y\|_2^2).
\end{eqnarray*}
\end{proof}

Thus as long as $\|y\|_2\ge \Omega(\frac{1}{n}),$ %$\frac{\|\nabla\Phi(x)\|_2}{\Phi(x)}$ 
gradient descent will decrease the potential function by a factor of $e^{-\Theta(1/n^2)}$ in each iteration, and so in at most $O(n^2\ln(m))$ iterations we arrive at a point
$x$ with $\Phi(x)<1.$%\implies x\in P$.%
% and so it won't take too long to get $\Phi(x)$ down past $1$.
%\footnote{For the classical analysis, we know that since $y$ is a convex combination of the constraint vectors we have $\|y\|_2\ge \rho$. (See Lemma 1 from \cite{smoothperceptronSP}). So we see that this method terminates in $\tilde{O}(1/\rho^2)$ iterations.
%%}
\footnote{Consider a single phase of the algorithm without rescaling. There exists $x^* \in P$ with $\|x^*\|_2 = 1$ so that $B(x^*,\rho) \subseteq P$, where $B(x^*,\rho) := \{ x \in \setR^n \mid \|x^* - x\|_2 \leq \rho\}$.
Then $\|\lambda A\|_2 \cdot \|x^*\|_2 \geq \left<\lambda A, x^*\right> = \sum_{i=1}^m \lambda_i \left<A_i,x^*\right> \geq \rho$,
since $\left<A_i,x^*\right> \geq \rho$ for all $i$. Therefore the algorithm is guaranteed to find a feasible point in $O(\frac{\ln(m)}{\rho^2})$ iterations without rescaling.
This argument is closely related to the classical analysis of the perceptron.}
%This gives an algorithm for linear programming with running time $\tilde{O}()$... something something something

\subsection{Modified Gradient Descent}
With $\Delta=\Theta(\frac{1}{n})$, the standard gradient descent approach implements the initial phase of Algorithm \ref{alg:fullalg} in $\tilde{O}(n^2)$ iterations. It turns out we can get the same guarantee in $\tilde{O}(n)$ iterations by choosing a more sophisticated update direction. 
%the number of iterations needed to do this down to $\tilde{O}(n)$. We do this using a slightly more balanced approach.
While we do not know how to guarantee an update direction that decreases 
$\Phi(x)$ by factor of more than $e^{-\Theta(\|y\|^2)}$, we are able to find a
direction so that the \emph{product} of $\Phi(x)$ and $\|\nabla \Phi(x)\|_2$ decreases
a lot faster. Note that in the following we will work with a general norm so that the results can be applied directly to either Algorithm \ref{alg:fullalg} (with $H=I$) or Algorithm \ref{alg:fullalgH}. We assume now that $\|A_i\|_{H^{-1}}=1$ for all $i$.

\begin{theorem}\label{thm:mwustep} Suppose $H\succ 0$ and $\|\lambda A\|_{H^{-1}}\ge \frac{\beta}{n}$, where $\beta>0$ is an arbitrary constant. Then in time $O(mn^{\omega-1})$, we can find $\varepsilon>0$ and $p \in \setR^n$ so that 
$$\hspace{.5cm}   \|\nabla \Phi(x+\varepsilon p)\|_{H^{-1}}\cdot\Phi(x+\varepsilon p)\le \|\nabla \Phi(x)\|_{H^{-1}}\cdot\Phi(x) \cdot e^{-\tilde{\Theta}(1/n)}. $$ 
\end{theorem}

Before going through the proof, we note that the update step of Theorem~\ref{thm:mwustep} yields a MWU phase that runs in time $\tilde{O}(mn^\omega)$. In particular, this gives the running time guarantee of Theorem \ref{thm:maintheoremmwu}.

% but we first want to show a quick consequence: \rem{B: Make the consequence into a theorem about the full algorithm}
%Note that we really are trying to decrease both $\Phi(x)$ and $\|y\|_2=\frac{\|\nabla\Phi(x)\|_2}{\Phi(x)}$, and so it suffices to find $\varepsilon, p$ satisfying 
%$$(*)\hspace{.5cm}   \|\nabla \Phi(x+\varepsilon p)\|_2\cdot\Phi(x+\varepsilon p)\le \|\nabla \Phi(x)\|_2\cdot\Phi(x) e^{-\Theta(\|y\|_2)}. $$ 
%Recall that to first order we have $\nabla \Phi(x+\varepsilon p)\approx \nabla \Phi(x) +\varepsilon \nabla^2\Phi(x)p$, and so
%$$\|\nabla \Phi(x+\varepsilon p)\|^2\approx \|\nabla \Phi(x)\|^2+\langle \nabla \Phi(x),\nabla^2\Phi(x)p\rangle +\|\nabla^2\Phi(x)p\|^2.$$
%Equivalently, $\|y\|_2\cdot \Phi(x)^2$ goes down by a factor of $e^{-\Theta(\|y\|_2)}$, and as soon as this quantity drops below $\Theta(\frac{1}{n})$ we must have either $\|y\|_2\le \Theta(\frac{1}{n})$ or $\Phi(x)<1$.
%Our main result of this section then follows immediately.

%\begin{lemma}\label{lem:mwuphaseH}
%Suppose $H$ is a symmetric positive definite matrix. In time $\tilde{O}(mn^{\omega})$ we can either find $x\in P$ or find a convex combination $\lambda$ with $\|\lambda A\|_{H^{-1}}\le O(\frac{1}{n})$.
%%\frac{\beta}{n}$, where $\beta$ an appropriate constant.
%\end{lemma}

\begin{lemma}\label{lem:mwuphase}
Suppose $H\succ 0$, and let $\beta$ be an arbitrary constant. Then in time $\tilde{O}(mn^\omega)$ we can run a MWU phase, which either finds $x\in P$ or gives $\lambda\in \setR^m_{\ge 0}$ with $\|\lambda\|_1=1$ and $\|\lambda A\|_{H^{-1}}\le \frac{\beta}{n}$.
\end{lemma}

\begin{proof}
Let $\lambda\ge 0$ be such that $\lambda A=-\frac{\nabla\Phi(x)}{\Phi(x)}$. Then as long as $\|\lambda A\|_{H^{-1}}\ge \frac{\beta}{n}$, Theorem~\ref{thm:mwustep} says that the quantity $\|\lambda A\|_{H^{-1}}\cdot \Phi(x)^2$ decreases by a factor of $e^{-\tilde{\Theta}(1/n)}$.
Then in $\tilde{O}(n)$ iterations we will have 
$\|\lambda A\|_{H^{-1}}\cdot \Phi(x)^2\le \frac{\beta}{n}$, which implies that either $\Phi(x)<1$ or $\|\lambda A\|_{H^{-1}}\le \frac{\beta}{n}$. 
\end{proof}

The remainder of this section will be devoted to the proof of Theorem \ref{thm:mwustep}.
%\begin{theorem} Suppose $\|\lambda A\|_{H^{-1}}\ge \frac{\beta}{n}$. Then in time $O(mn^{\omega-1})$, we can find $\varepsilon>0$ and $p \in \setR^n$ so that 
%$$\hspace{.5cm}   \|\nabla \Phi(x+\varepsilon p)\|_{H^{-1}}\cdot\Phi(x+\varepsilon p)\le \|\nabla \Phi(x)\|_{H^{-1}}\cdot\Phi(x) \cdot e^{-\tilde{\Theta}(1/n)}. $$ 
%\end{theorem}
We begin by establishing some useful notation.
For any symmetric positive definite matrix $H\succ 0$ we define the inner product $\left<x,y\right>_H := x^THy$. Without any subscript $\left<x,y\right> = x^Ty$ will continue to denote the canonical inner product.

Given $x\in \setR^n$, define $\lambda_i=\frac{1}{\Phi(x)}e^{-\langle A_i,x\rangle}$, $y=-\frac{\nabla \Phi(x)}{\Phi(x)}=\sum_{i=1}^m\lambda_iA_i$ and $M=\frac{\nabla^2\Phi(x)}{\Phi(x)}=\sum_{i=1}^m\lambda_iA_iA_i^T$. Even though all three depend on $x$, we will not denote that here to keep the notation clean. 

To prove Theorem \ref{thm:mwustep}, we first show how $\Phi(x)$ decreases as we take steps in an arbitrary direction $p$.
%As one might expect, these bounds will involve the gradient and the Hessian of $\Phi$, and hence also $y$ and $M$. 

%The analysis the MWU phase of Algorithm \ref{alg:fullalgH} is similar to that of Algorithm \ref{alg:fullalg}, but using different norms.
\begin{lemma}\label{lem:phidecreaseH}
For any $0<\varepsilon\le 1$ and $p\in \R^n$ with $\|p\|_H\le 1$, we have 
%$$\Phi(x+p)\le \Phi(x)+\langle \nabla \Phi(x),p\rangle +p^T(\nabla^2\Phi(x))p.$$
 $$\Phi(x+\varepsilon p)\le \Phi(x)\cdot (1-\varepsilon \langle y,p\rangle +\varepsilon^2 p^TMp).$$
\end{lemma}
\begin{proof}
Notice that since $\|p\|_H\le 1$ and $\|A_i\|_{H^{-1}}=1$ we have $|\langle A_i,\varepsilon p\rangle|\le 1$ by the generalized Cauchy-Schwarz inequality.
Writing out the definitions we obtain
\begin{eqnarray*}
\Phi(x+\varepsilon p) &=& \sum_{i=1}^m e^{-\langle A_i,x+\varepsilon p\rangle}\\
&=& \sum_{i=1}^m e^{-\langle A_i,x\rangle }e^{-\varepsilon\langle A_i,p\rangle}\\
&\stackrel{(*)}{\le} & \sum_{i=1}^m e^{-\langle A_i,x\rangle}(1-\varepsilon \langle A_i,p\rangle +\varepsilon^2\langle A_i,p\rangle ^2)\\
&=&\Phi(x)\cdot\sum_{i=1}^m\lambda_i(1-\varepsilon \langle A_i,p\rangle +\varepsilon^2 p^TA_iA_i^Tp)\\
%&=& \sum_{i=1}^m e^{-\langle A_i,x\rangle }-\varepsilon \langle \sum_{i=1}^m e^{-\langle A_i,x\rangle} A_i,p\rangle+\varepsilon^2 p^T\sum_{i=1}^m e^{-\langle A_i,x\rangle}A_iA_i^Tp\\
%&=& \Phi(x)+\varepsilon \langle \nabla\Phi(x),p\rangle +\varepsilon^2p^T\nabla^2\Phi(x)p\\
&=& \Phi(x)\cdot (1-\varepsilon \langle y,p\rangle +\varepsilon^2 p^TMp).
\end{eqnarray*}
In $(*)$ we use the estimate that for any $z\in \R$ with $|z|\le 1$ one has $e^z\le 1+z+z^2$.
\end{proof}

In a similar way, we bound $\|\nabla \Phi(x)\|_{H^{-1}}$ after an update step in an arbitrary direction $p$.

\begin{lemma}\label{lem:gradphidecreaseH} Suppose $p\in\R^n$ with $\|p\|_H\le 1$, and $0<\varepsilon\le 1$
%\in \R^n$ 
we have
\begin{eqnarray*}
\|\nabla \Phi(x+ \varepsilon p)\|_{H^{-1}}\le \|\nabla \Phi(x)\|_{H^{-1}}\cdot 
\left( 1 + \frac{1}{\|y\|_{H^{-1}}} \varepsilon^2\langle p,Mp\rangle +\frac{1}{\|y\|_{H^{-1}}^2}\left(-\varepsilon\langle y,Mp\rangle_{H^{-1}}+\varepsilon^2 \|Mp\|_{H^{-1}}^2\right)\right)
\end{eqnarray*}
%\left(1+\frac{1}{2\|y\|_{H^{-1}}^2}\left(-\varepsilon \langle y,Mp\rangle_{H^{-1}}+\varepsilon^2\|Mp\|_{H^{-1}}^2\right)\right) +\varepsilon^2\langle p,Mp\rangle$$
\end{lemma}

\begin{proof}
For any $z$ with $|z|\le 1$, we have $e^z=1+z+\eta z^2$ for some $\eta \in \R$ with $|\eta|\le 1$.
In particular, since $\|A_i\|_{H^{-1}}=1$ and $\|p\|_H\le 1$, we have $|\langle A_i,\varepsilon p\rangle|\le 1$ and so we have such an $\eta_i$ for each $i$.
\begin{eqnarray*}
\frac{\|\nabla\Phi(x+\varepsilon p)\|_{H^{-1}}}{\Phi(x)}
&=& \frac{1}{\Phi(x)}\Big\|\sum_{i=1}^m (-A_i)\cdot e^{-\langle A_i,x+\varepsilon p\rangle}\Big\|_{H^{-1}}\\
&=&  \frac{1}{\Phi(x)}\Big\|\sum_{i=1}^m(-A_i)e^{-\langle A_i,x\rangle}e^{-\varepsilon \langle A_i,p\rangle}\Big\|_{H^{-1}}\\
&=& \Big\| \sum_{i=1}^m (-A_i) \cdot \lambda_i \cdot \Big(1-\varepsilon \left<A_i,p\right> + \varepsilon^2 \cdot \eta_i \left<A_i,p\right>^2\Big) \Big\|_{H^{-1}} \\
%&\stackrel{\textrm{triangle ineq}}{\leq}& 
&\le &\Big\| \sum_{i=1}^m (-A_i) \cdot \lambda_i \cdot (1-\varepsilon \left<A_i,p\right>) \Big\|_{H^{-1}}
+ \varepsilon^2 \cdot \Big\| \sum_{i=1}^m (-A_i) \cdot \lambda_i \eta_i \left<A_i,p\right>^2 \Big\|_{H^{-1}}  \\ 
%&\stackrel{\textrm{triangle ineq}}
&{\leq}& \Big\| \sum_{i=1}^m \lambda_iA_i- \varepsilon \sum_{i=1}^m \lambda_iA_i\left<A_i,p\right>
\Big\|_{H^{-1}} + \varepsilon^2 \cdot \sum_{i=1}^m \underbrace{\|A_i\|_{H^{-1}}}_{=1} \cdot \lambda_i \cdot \underbrace{|\eta_i|}_{\leq 1} \cdot \left<A_i,p\right>^2 \\
&\leq& \Big\| y - \varepsilon \cdot Mp \Big\|_{H^{-1}} + \varepsilon^2 \cdot \sum_{i=1}^m \lambda_i\left<A_i,p\right>^2 \\
 &=& \left( \|y\|_{H^{-1}}^2-2\varepsilon y^TH^{-1}Mp +\varepsilon^2 \|Mp\|_{H^{-1}}^2\right)^{1/2}+\varepsilon^2p^TMp\\
&\le & \|y\|_{H^{-1}} \cdot \left(1+2 \frac{1}{\|y\|_{H^{-1}}^2} \left(-\varepsilon\langle y,Mp\rangle_{H^{-1}} 
+\varepsilon^2 \|Mp\|_{H^{-1}}^2\right)\right)^{1/2}+\varepsilon^2p^TMp\\
& \le & \|y\|_{H^{-1}}\left( 1 +\frac{1}{\|y\|_{H^{-1}}^2}\left(-\varepsilon\langle y,Mp\rangle_{H^{-1}}+\varepsilon^2 \|Mp\|_{H^{-1}}^2\right)\right) +\varepsilon^2p^TMp\\
& \le & \|y\|_{H^{-1}}\left( 1 + \frac{1}{\|y\|_{H^{-1}}}\varepsilon^2p^TMp+\frac{1}{\|y\|_{H^{-1}}^2}\left(-\varepsilon\langle y,Mp\rangle_{H^{-1}}+\varepsilon^2 \|Mp\|_{H^{-1}}^2\right)\right)
\end{eqnarray*}
Recalling that $\nabla\Phi(x)=-\Phi(x)y$ finishes the proof. 
\end{proof}

%As in Lemma \ref{lem:phigradphidecrease}, 
Using Lemmas \ref{lem:phidecreaseH} and \ref{lem:gradphidecreaseH}, we can show a sufficient condition for $p$ to satisfy Theorem \ref{thm:mwustep}.

\begin{lemma}\label{lem:phigradphidecreaseH}
 Suppose $p\in \R^n$ with $\|p\|_H\le 1$ and constant $a>0$ is such that either
\begin{enumerate}
\item $\langle y,p\rangle \ge \frac{\|y\|_{H^{-1}}}{(\log{n})^a}$ and $\langle y,Mp\rangle_{H^{-1}} \ge \frac{\|Mp\|_{H^{-1}}\cdot \|y\|_{H^{-1}}}{(\log{n})^a}$ or 
\item $\langle y,p \rangle \ge \frac{\|y\|_{H^{-1}}}{(\log{n})^a}$ and $\|Mp\|_{H^{-1}}\le O\left( \frac{1}{\poly(n)}\right)$.
\end{enumerate}
Then as long as $\|y\|_{H^{-1}}\ge \frac{\beta}{n}$, choosing $\varepsilon =\min \left\{\frac{\|y\|_{H^{-1}}}{4(\log{n})^{2a}\|Mp\|_{H^{-1}}},\frac{1}{2(\log{n})^a}\right\}$ gives 
$$ \|\nabla \Phi(x+ \varepsilon p)\|_{H^{-1}}\cdot\Phi(x+ \varepsilon p)\le \|\nabla \Phi(x)\|_{H^{-1}}\cdot\Phi(x) e^{-\tilde{\Theta}(1/n)}. $$ 
\end{lemma}

\begin{proof}
Let $\varepsilon =\min \left\{\frac{\|y\|_{H^{-1}}}{4(\log{n})^{2a}\|Mp\|_{H^{-1}}},\frac{1}{2(\log{n})^a}\right\}$. 
Then by Lemma \ref{lem:phidecreaseH}, we have
$$\Phi(x+\varepsilon p)\le \Phi(x)\cdot (1-\varepsilon \langle y,p\rangle +\varepsilon^2\|Mp\|_{H^{-1}})\le \Phi(x) \cdot (1-\varepsilon \frac{\|y\|_{H^{-1}}}{(\log{n})^a}+\varepsilon^2\|Mp\|_{H^{-1}}).$$

Assume first that we are in Case $1$.
By Lemma \ref{lem:gradphidecreaseH}, since we know $\varepsilon\le \frac{1}{2(\log{n})^a}$, we have
$$\|\nabla\Phi(x+\varepsilon p)\|_{H^{-1}}\le \|\nabla\Phi(x)\|_{H^{-1}}\cdot \Big(1-\frac{\varepsilon}{2} 
\frac{\|Mp\|_{H^{-1}}}{\|y\|_{H^{-1}}(\log{n})^{a}} +\varepsilon^2 \frac{\|Mp\|_{H^{-1}}^2}{\|y\|_{H^{-1}}^2}\Big).$$

If $\varepsilon=\frac{1}{2(\log{n})^a}$, then $\|Mp\|_{H^{-1}} \le \frac{\|y\|_{H^{-1}}}{2(\log{n})^a}$.
%$\|y\|\ge \frac{2\|Mp\|}{(\log{n})^a}$, 
%and we see that $\Phi$ will decrease by $(1-\frac{\|y\|}{2(\log{n})^a})$.
Using this, $\Phi(x)$ will decrease by $e^{-\tilde{\Theta}(\|y\|_{H^{-1}})}$, and 
$\|\nabla\Phi(x)\|_{H^{-1}}$ will decrease.
%we have $\|\nabla \Phi(x+\varepsilon p\|_{H^{-1}}\le \|\nabla \Phi(x)\|_{H^{-1}}$.

On the other hand, if $\varepsilon=\frac{\|y\|_{H^{-1}}}{4(\log{n})^{2a}\|Mp\|_{H^{-1}}}$, then $\Phi(x)$ will decrease, and $\|\nabla\Phi(x)\|_{H^{-1}}$ decreases by $e^{-\tilde\Theta(1)}$. 
%Similarly, $\|\nabla \Phi\|_{H^{-1}}$ decreases by 
%$$\left(1-\frac{\|Mp\|}{\|y\|}\varepsilon+\varepsilon^2\frac{\|Mp\|^2}{\|y\|^2}\right),$$ and so no matter what it does not increase by more than $(1+\frac{1}{\poly(n)})$.
%If $\varepsilon =\frac{\|y\|}{2\|Mp\|}$, then in fact $\nabla \Phi$ will decrease by $(1-\frac{1}{4})$.
%$\|y\|\le \frac{2\|Mp\|}{(\log{n})^a}$, and so 
Together, these show that the product decreases by a factor of $e^{-{\tilde{\Theta}}(\|y\|_{H^{-1}})}$.

If we are in Case $2$, the only thing that might change is that when $\varepsilon=\frac{1}{2(\log{n})^a}$ we might have $\|\nabla\Phi(x)\|_{H^{-1}}$ increase by up to $e^{O(1/\poly(n))}$. Since $\|y\|_{H^{-1}}\ge \frac{\beta}{n}$ this is the dominating term, and so we will still get the appropriate decrease. 
\end{proof}

%-------------------------------------------------------------------------------------------------------------------------------%

Notice that the conditions of Lemma \ref{lem:phigradphidecreaseH} essentially say that both $p$ and $Mp$ are close in angle with the vector $y$.
In particular, if the gradient happened to be an eigenvector of the Hessian (and hence $y$ an eigenvector of $M$) then Lemma \ref{lem:phigradphidecreaseH} would be satisfied with $p=\frac{y}{\|y\|_{H^{-1}}}$. With this in mind, the idea for computing such a direction $p$ is to project $y$ onto an appropriate eigenspace of $M$.
%
%Ignoring $\log$ terms for a moment and letting $\gamma=\|Mp\|_{H^{-1}}$ the two lemmas above give us
%$$\Phi(x+\varepsilon p)\le \Phi(x)\left( 1-\varepsilon \|y\|_{H^{-1}} +\varepsilon^2\gamma\right)$$ and 
%$$\|\nabla\Phi(x+\varepsilon p)\|_{H^{-1}}\le \|\nabla \Phi(x)\|_{H^{-1}}\left(1-\varepsilon\frac{\gamma}{\|y\|_{H^{-1}}}+\varepsilon^2\frac{\gamma^2}{\|y\|_{H^{-1}}^2}\right)$$
%
%Letting $\varepsilon=\min\{\frac{\|y\|_{H^{-1}}}{\gamma},1\}$, we get that they both have to go down and one must go down by enough. (We'll have to introduce log factors back in)
%To prove Lemma \ref{lem:computingpH}, 
To this end, we first prove the following general statement about computing approximate eigenvectors of matrices.

\begin{lemma}\label{lem:computingplemma}
Suppose $z\in \mathbb{R}^n$ is a unit vector, $N$ a PSD matrix with no eigenvalue bigger than $1$ and $K>0$ a given parameter.
For $k=1,...,K$, define $z_k=(I-N)^{2^k}z$. Then $\|z_k\|_2\le 1$, and for an appropriate constant $C>0$ at least one of the following must hold:
\begin{enumerate}
\item There exists $k\le K$ with
$$\langle z,z_k\rangle\ge \frac{C}{K} \textrm{ and } 
\langle z,Nz_k\rangle \ge \frac{C\|Nz_k\|_2}{K^2}$$
\item For $k=K$, we have $$\langle z,z_K\rangle\ge \frac{C}{K} \textrm{ and }
\|Nz_K\|_2\le \frac{K}{2^K}$$
%=O(\frac{1}{\poly(n)})$$
\end{enumerate}
\end{lemma}

\begin{proof}
First note that we may assume in fact that no eigenvalue of $N$ is bigger than $\frac{1}{2}$ since we can always replace $N$ with $\frac{1}{2}N$ and only lose a factor of $2$.
Suppose now the eigenvectors of $N$ are unit vectors $v_1,...,v_n$ with eigenvalues $\alpha_1,...,\alpha_n$.

We see that $$z_k=\sum_{j=1}^n (1-\alpha_j)^{ 2^k}\langle z,v_j\rangle v_j.$$%\approx\sum e^{-\alpha_j \beta_k}c_jv_j.$$
$$Nz_k=\sum_{j=1}^n\alpha_j(1-\alpha_j)^{2^k}\langle z,v_j\rangle v_j.$$%\approx \sum \alpha_je^{-\alpha_j\beta_k}c_jv_j.$$
$$\langle z_k,z\rangle =\sum_{j=1}^n(1-\alpha_j)^{2^k}\langle z,v_j\rangle ^2$$
\noindent For any $k$, we can get the following bounds.

\begin{itemize}
%\item Since $(1-\alpha_j)^{\beta_k}\le 1$ for all $j$, we have that $\|y_k\|\le 1$.
\item For any $\alpha_j$, we either have $\alpha_j\le \frac{k}{2^k}$ or $e^{-\alpha_j2^k}\le e^{-k}\le 2^{-k}$. In either case we have the bound $\alpha_j(1-\alpha_j)^{2^k}\le\alpha_je^{-\alpha_j2^k}\le \frac{k}{2^k}$.
Therefore we can conclude that $\|Nz_k\|\le\frac{k}{2^k}$.
\item Whenever $\alpha_j\le \frac{1}{2^k}$
%\in[\frac{1}{2\beta_k},\frac{1}{\beta_k}]$
 we have $(1-\alpha_j)^{2^k}\ge e^{-2\alpha_j2^k}\ge \frac{1}{e^2}$ and all other coefficients will be nonnegative, and therefore
 $$\langle z,z_k \rangle\ge \frac{1}{e^2}\sum_{\alpha_j\le 2^{-k}} \langle z,v_j\rangle ^2$$
% $\langle y,y_k \rangle\ge \frac{1}{\polylog(m,n)}.$
\end{itemize}
Now, notice that since $\sum_{j=1}^n\langle z,v_j\rangle^2=\|z\|^2$, either there exists $k\le K$ so that 
$$\sum_{\alpha_j\in\left[\frac{1}{2^{k+1}},\frac{1}{2^k}\right]}\langle z,v_j\rangle^2\ge \frac{1}{2K},$$
%where $\beta_k=2^{k}$. 
or else we must have $$\sum_{\alpha_j\le 2^{-K}}\langle z,v_j\rangle^2\ge \frac{1}{2} \implies \langle z,z_K\rangle \ge \frac{1}{2e^2}$$ 
In the latter case we are done, since we already showed $\|Nz_k\|\le\frac{K}{2^K}$.

Otherwise, choose this $k$, and notice that whenever $\alpha_j\in[\frac{1}{2^{k+1}},\frac{1}{2^k}]$, we have  $\alpha_j(1-\alpha_j)^{2^k}\ge \alpha_je^{-2\alpha_j2^k}\ge \frac{1}{e^22^{k+1}}$ , and all other coefficients will be nonnegative.
Therefore $\langle z,Nz_k\rangle = \sum_{j=1}^n \alpha_j(1-\alpha_j)^{2^k}\langle z,v_j\rangle^2 \ge \frac{1}{e^22^{k+1}K}$ and so in particular 
$\langle z,Nz_k\rangle \ge \frac{\|Nz_k\|_2}{2e^2K^2},$
as desired. 
%In the second, the point is that either $\alpha_j\le \frac{\polylog(n)}{2^k}$, in which case the statement is clear, or else $\alpha_j\ge \frac{\polylog(n)}{\beta_k}$, in which case the term is bounded by $\frac{1}{\poly(n)}$, so it is certainly small enough.
\end{proof} 

%For Lemma \ref{lem:computingp}, we set $z=\frac{y}{\|y\|_2}$ and $N=\frac{1}{2}M$.

Finally, using Lemma \ref{lem:computingplemma}, we show that we can efficiently compute a vector $p$ satisfying Lemma \ref{lem:phigradphidecreaseH}, and hence complete the proof of Theorem \ref{thm:mwustep}.

\begin{lemma}\label{lem:computingpH} 
In time $\tilde{O}(mn^{\omega-1})$ we can find $p\in \R^n$ satisfying the hypotheses of Lemma \ref{lem:phigradphidecreaseH}.
%If $z_k=(I-M)^{2^k}\frac{y}{\|y\|}$ then for some $k\le K=\Theta(\log{n})$, $p=z_k$ satisfies the hypotheses of Lemma \ref{lem:phigradphidecreaseH}.
\end{lemma} 

%To prove Lemma \ref{lem:computingpH}, we s
\begin{proof}
Set $K=10\log{n}$, $z=\frac{H^{-1/2}y}{\|y\|_{H^{-1}}}$ and $N=H^{-1/2}MH^{-1/2}$ in Lemma \ref{lem:computingplemma} to get output of $z_k$. 
%\rem{B: How specific should we be about what $K$ is? Also should we have $1/\poly(n)$ or actually pick an explicit polynomial? (both for lemmas 16 and 17)}
Let $p=H^{-1/2}z_k$ and notice that 
\begin{enumerate}
\item $\|p\|_H=\|z_k\|_2$,
\item $\|Mp\|_{H^{-1}}=\|Nz_k\|_2$
\item $\langle \frac{y}{\|y\|_{H^{-1}}},p\rangle=\langle z,z_k\rangle$
\item $\langle \frac{y}{\|y\|_{H^{-1}}},Mp\rangle_{H^{-1}}=\langle z,Nz_k\rangle$
\end{enumerate}
% and so $\langle p,Mp\rangle \le \|Mp\|_{H^{-1}}$.
In particular, rearranging the statement of Lemma \ref{lem:computingplemma},
this $p$ satisfies the hypotheses of Lemma \ref{lem:phigradphidecreaseH}. 
Finally, note that computing $M$ takes time $O(n^{\omega-1})$ and all other matrix operations can be computed in time $O(n^\omega)$. Since we perform at most $O(\log{n})$ iterations, the running time is $\tilde{O}(mn^{\omega-1})$, as desired.
\end{proof}
%\begin{enumerate}
%\item $\langle y,p\rangle \ge \frac{\|y\|_{H^{-1}}}{\polylog(n)}$
%and
%$\langle y,Mp\rangle_{H^{-1}}\ge \frac{\|Mp\|_{H^{-1}}\|y\|_{H^{-1}}}{\polylog(n)}$
%\item $\langle y,p\rangle \ge \frac{\|y\|_{H^{-1}}}{\polylog(n)}$ and $\|Mp\|_{H^{-1}}\le O\left(\frac{1}{\poly(n)}\right)$
%\end{enumerate}
%To prove Lemma \ref{lem:computingp}, we set $H=I$ above. In other words, we set $N=M$ and $z=\frac{y}{\|y\|_2}$ and let $p=z_k$ for some $k$.

%Finally, we show that we are able to compute such a direction $p$ efficiently.

%Combining Lemmas \ref{lem:phigradphidecreaseH} and \ref{lem:computingpH} finishes the proof of Theorem \ref{thm:mwustep}.% and hence of Lemma \ref{lem:mwuphaseH}.

\section{Computing an Approximate John Ellipsoid}

It turns out that our algorithm implicitly computes an approximate John ellipsoid for the considered cone $P$, which gives us geometric insight into $P$.
Recall that a classical theorem of John \cite{johnellipsoid} shows that for any closed, convex set $Q \subseteq \setR^n$, 
there is an ellipsoid $E$ and a center $z$ so that $z + E \subseteq Q \subseteq z + nE$. The bound of $n$
is tight in general --- for example for a simplex --- but it can be improved to $\sqrt{n}$ for symmetric sets.
This is equivalent to saying that for each convex body, there is a linear transformation that makes it
well \emph{$n$-well rounded}. Here, a body $Q$ is $\alpha$-well rounded if $z+r \cdot B \subseteq Q \subseteq z+\alpha\cdot r \cdot B$
for some center $z \in \setR$ and radius $r>0$.
See the excellent survey of Ball~\cite{AnElementaryIntroToConvexGeometry-Ball97} on this topic.
%If for some factor $\rho$ one has $z + B \subseteq Q \subseteq z + \rho B$, then one says that  $Q$ is 
%\emph{$\rho$-well-rounded}.

A summary of our full MWU algorithm with rescaling is given in Algorithm \ref{alg:fullalgdetailed}.
We will prove here that after a minor modification of the algorithm, the set $P \cap B$ 
will be well rounded when the algorithm terminates. 

\begin{algorithm}
\caption{\label{alg:fullalgdetailed}}%Our Algorithm}
\noindent FOR $O(n\log(\frac{1}{\rho}))$ phases DO
\begin{itemize} \vspace{-2mm}
\item {\bf MWU phase:} 
  \begin{enumerate}[nolistsep]
  \item[(1)] Normalize $\|A_i\|_2=1$ for all $i$, and set $x^{(0)} := \bm{0}$ 
  \item[(2)] FOR $t := 0$ TO $T$ DO
     \begin{enumerate}[nolistsep]
     \item[(3)] Set $\lambda_i^{(t)} := \frac{1}{\Phi(x^{(t)})} \exp(-\left<A_i,x^{(t)}\right>)$
     \item[(4)] If $\Phi(x^{(t)}) < 1$ THEN RETURN $x^{(t)} \in P$
     \item[(5)] IF $\|\lambda^{(t)}A\|_2 \leq \frac{\beta}{n}$ THEN GOTO Rescaling phase
     \item[(6)] Select an \emph{update vector} $p^{(t)} \in \setR^n$ with $\|p^{(t)}\|_2 \leq 1$
     \item[(7)] Select a \emph{step size} $0<\varepsilon_t \leq 1$
     \item[(8)] Update $x^{(t+1)} := x^{(t)} + \varepsilon_tp^{(t)}$ \vspace{-2mm}
     \end{enumerate}
  \end{enumerate} 
%  Either return $x\in P$ or give a convex combination of constraints with $\|\sum_{i=1}^m\lambda_iA_i\|_2\le \frac{\beta}{n}$.
\vspace{2mm}
\item {\bf Rescaling phase:}  
  \begin{enumerate}[nolistsep]
  \item[(1)] Compute an invertible linear transformation $F$ so that $\vol(F(P)\cap B)$ is a constant factor larger than $\vol(P\cap B)$. Replace $P$ by $F(P)$. \vspace{-2mm} 
  \end{enumerate}
\end{itemize}
\end{algorithm}

%of the cone 
%with the unit ball at the end
%Of course, the cone $P$ in our case is unbounded, so our notion of approximation will be 
%somewhat different. (explain this more)

\begin{lemma}
Consider Algorithm \ref{alg:fullalgdetailed} with the modification that the MWU phase terminates in step (4) only if $\Phi(x) < \frac{1}{e}$. 
Then in the final iteration $P \cap B$ is $\tilde{O}(n)$-well rounded. 
\end{lemma}
\begin{proof}
Let us consider the last phase of the algorithm and let $x^{(0)},\ldots,x^{(T)}$ be the computed sequence of points 
with $\Phi(x^{(T)}) < \frac{1}{e}$ and $T \leq \tilde{O}(n)$. Then $e^{-\left<A_i,x^{(T)}\right>} < \frac{1}{e}$ and hence
$\left<A_i,x^{(T)}\right> \geq 1$ for all $i$. The step size of the algorithm is always bounded by $1/2$, hence
$\|x^{(T)}\|_2 \leq \frac{T}{2}$. Now define $z := \frac{1}{T} \cdot x^{(T)}$ as center. Then $\|z\|_2 \leq \frac{1}{2}$
and $\left<A_i,z\right> \geq \frac{1}{T}$. Hence $B(z,\frac{1}{T}) \subseteq P \cap B \subseteq B(z,1)$, which shows 
that $P\cap B$ is $T$-well rounded.
\end{proof}
Note that running the algorithm until $\Phi(x) < \frac{1}{e}$
only increases the worst case running times by a constant factor. 
Alternatively one can run the algorithm with standard gradient descent and a fixed step size of $\varepsilon := \Theta(\frac{1}{n})$ and 
only terminate when $\Phi(x) < \frac{1}{m}$. This increases the running time by up to a factor of $n$, but the final set $P \cap B$ will be $O(n)$-well rounded, thus removing the logarithmic terms suppressed by the $\tilde{O}$ notation. 
On the other hand, no linear transformation can make the conic hull of a simplex $o(n)$-well rounded, 
hence our obtained bound is asymptotically optimal. Note that to obtain the 
tight factor for well-roundedness it was crucial to have the optimal 
rescaling threshold of $\Delta = \Theta(\frac{1}{n})$.
%The situation is visualized in Figures \ref{fig:balls} and \ref{fig:ellipsoids}. 

\paragraph{Independent publication.}

The multi-rank rescaling was also discovered in a parallel and independent work by Dadush, Vegh and Zambelli~\cite{DBLP:journals/corr/DadushVZ16} (see their Algorithm~5).

\bibliographystyle{alpha}
\bibliography{mwu-lp}

\end{document}